\documentclass[12pt,amscd]{amsart}
\usepackage[all]{xy}
\usepackage{graphicx,tikz}
\usepackage{amsmath,amsxtra,amssymb,latexsym, amscd,amsthm}
\usepackage{indentfirst}
\usepackage[mathscr]{eucal}
\usepackage[pagebackref=true]{hyperref}

\newtheorem{thm}{Theorem}[section]
\newtheorem{cor}[thm]{Corollary}
\newtheorem{lem}[thm]{Lemma}
\newtheorem{prop}[thm]{Proposition}
\theoremstyle{definition}
\newtheorem{defn}[thm]{Definition}
\newtheorem{exm}[thm]{Example}

\newtheorem{rem}[thm]{Remark}

\numberwithin{equation}{section}

\DeclareMathOperator{\NN}{\mathbb {N}}

\DeclareMathOperator{\RR}{\mathbb {R}}

\DeclareMathOperator{\dstab}{dstab}

\DeclareMathOperator{\depth}{depth}
\DeclareMathOperator{\pd}{pd}
\DeclareMathOperator{\supp}{supp}

\def\B {\mathcal B}

\def\e {\mathbf e}
\def\a {\mathbf a}
\def\b {\mathbf b}
\def\t {\mathbf t}

\def\m {\mathfrak m}

\def\k {\mathrm{k}}

\begin{document}

\title[Depth function of cycles] {Depth of powers of edge ideals of cycles and trees}

\author{Nguyen Cong Minh}
\address{School of Applied Mathematics and Informatics, Hanoi University of Science and Technology, 1 Dai Co Viet, Hanoi, Vietnam}
\email{minh.nguyencong@hust.edu.vn}

\author{Tran Nam Trung}
\address{Institute of Mathematics, VAST, 18 Hoang Quoc Viet, Hanoi, Vietnam}
\email{tntrung@math.ac.vn}

\author{Thanh Vu}
\address{Institute of Mathematics, VAST, 18 Hoang Quoc Viet, Hanoi, Vietnam}
\email{vuqthanh@gmail.com}


\date{}

\dedicatory{Dedicated to Professor Ngo Viet Trung on the occasion of his 70th birthday}
\commby{}
\maketitle
\begin{abstract}
    Let $I$ be the edge ideal of a cycle of length $n \ge 5$ over a polynomial ring $S = \k[x_1,\ldots,x_n]$. We prove that for $2 \le t < \lceil (n+1)/2 \rceil$, 
    $$\depth (S/I^t) = \lceil \frac{n -t + 1}{3} \rceil.$$
    
    When $G = T_\a$ is a starlike tree which is the join of $k$ paths of length $a_1, \ldots, a_k$ at a common root $1$, we give a formula for the depth of powers of $I(T_\a)$.
\end{abstract}

\maketitle

\section{Introduction}
\label{sect_intro}
Let $I$ be a homogeneous ideal in a polynomial ring $S = \k[x_1,\ldots,x_n]$. Brodmann \cite{Br} proved that the depth function of powers $k \to \depth S/I^k$ is convergent. Ha, Nguyen, Trung, and Trung \cite{HNTT} proved that the depth function of powers of monomial ideals could be any non-negative integer-valued convergent function. On the other hand, when restricting to edge ideals of graphs, one expects that the depth function of their powers is non-increasing. This phenomenon has been verified for several classes of graphs (see \cite{HH, KTY, Mo}).

Let us now recall the notion of the edge ideals of graphs. Let $G$ be a simple graph on $n$ vertices. The edge ideal of $G$, denoted $I(G)$ is the squarefree monomial ideal generated by $x_ix_j$ where $\{i,j\}$ is an edge of $G$. For a homogeneous ideal $I$, we denote $\dstab(I)$ the {\it index of depth stability} of $I$, i.e., the smallest number $k$ such that $\depth S/I^\ell = \lim_{i \to \infty} \depth S/I^i$ for all $\ell \ge k$. In a fundamental paper \cite{T}, the second author found a combinatorial formula for $\dstab(I(G))$ for large classes of graphs including unicyclic graphs. In particular, when $G$ is a tree, $\dstab(I(G)) = n - \epsilon_0(G)$ where $\epsilon_0(G)$ is the number of leaves of $G$; when $G$ is a cycle of length $n \ge 5$, $\dstab(I(G)) = \lceil\frac{n+1}{2} \rceil$. Though we know the limit depth and its index of depth stability, intermediate values for depth of powers of edge ideals were unknown even for cycles. The depth of powers of edge ideals of paths was only given recently by Balanescu and Cimpoeas \cite{BC1}. For general graphs, Seyed Fakhari \cite{SF} gave a sharp lower bound for the depth of the second power of their edge ideals. In this paper, we compute the depth of powers of edge ideals of cycles.

\begin{thm}\label{depth_power_cycle}
    Let $I(C_n)$ be the edge ideal of a cycle of length $n \ge 5$. Then 
    $$\depth S/I(C_n)^t = \begin{cases} \lceil \frac{n-1}{3} \rceil & \text { if } t = 1\\
    \lceil \frac{n-t+1}{3} \rceil & \text { if } 2 \le t < \lceil \frac{n+1}{2}  \rceil\\
    1 & \text{ if } n \text{ is even and } t \ge  n/2 + 1\\
    0 & \text{ if } n \text{ is odd and } t \ge (n+1)/2.\end{cases}$$
\end{thm}
In particular, the depth function of powers of cycles makes a big drop just before it stabilizes. The initial value and limiting values were well-known (see \cite{Mo, BC2, T}), so our contribution is the computation of the intermediate values $\depth S/I(C_n)^t$ for $2\le t < \lceil (n+1)/2 \rceil$. We now outline the ideas to carry out this computation. For simplicity of notation, we set $I = I(C_n)$ and $e_i = x_i x_{i+1}$ for $i = 1, \ldots, n-1$, $e_n = x_1x_n$. Denote 
$$\varphi(n,t) = \lceil \frac{n-t+1}{3} \rceil.$$
\begin{enumerate}
    \item First, we show that $\depth S/ I^t \le \depth (S / (I^t: e_2 \cdots e_t)) \le \varphi(n,t)$.
    \item Let $f = x_1 \cdots x_{2t-2}$. By \cite[Theorem 4.3]{CHHKTT}, we need to show that $\depth (S/(I^t:f)) \ge \varphi(n,t)$ and $\depth (S/(I^t,f)) \ge \varphi(n,t)$. For the first term, we use induction on $t$ as $I^t:f$ is well-understood. For the second term, we note that $(I^t,f) = (I^t,x_1x_2) \cap (I^t,x_3 \cdots x_{2t-2})$. By repeated use of the depth lemma along a short exact sequence, we reduce to proving that $\depth (S/(I^t + I(H))) \ge \varphi(n,t)$ for all non-zero subgraph $H$ of $C_n$. We accomplish that by induction on $t$ and downward induction on the size of $H$.
\end{enumerate}
To compute $\depth S/I(G)^t$ for general graph $G$, by the result of Nguyen and the last author \cite[Theorem 1.1]{NV2}, we may assume that $G$ is connected. We also note that the regularity of powers of edge ideals of graphs has been known for many classes of graphs (see \cite{MV} for a recent survey on the topic). This is partly due to the fact that the regularity of powers of these edge ideals behaves nicely. On the other hand, the depth of powers of edge ideals was much more subtle as the drop at each step could be arbitrary, even for trees.

\begin{exm} Let $G$ be a caterpillar tree based on a path of length $2$ on $3$ vertices $1, \ldots, 3$. Assume that $d_i$ is the number of leaves at $i$. By Proposition \ref{prop_caterpillar}, we have 
$$\depth S/I(G)^t = \begin{cases} 
\min (d_1 + d_3+1,d_2+2) & \text{ if } t = 1\\
\min (d_1+1,d_2+2,d_3+1) & \text{ if } t = 2\\
1 & \text{ if } t \ge 3.\end{cases}$$      
Hence, if we choose $d_2 < d_1, d_3$ then $\depth S/I = \depth S/I^2 = d_2+2$ but then drop to $1$ for the next powers. If $d_1+d_3 < d_2$, the drop at the first step could be arbitrary.    
\end{exm}
We will now describe the formula for depth of powers of edge ideals of starlike trees. We first introduce some notation. Let $\a = (a_1, \ldots, a_k) \in \NN^k$ be a tuple of positive integers. Let $\alpha_i$ be the number of $a_j$ such that $a_j = i \pmod 3$. We define $g: \NN^k \to \NN$ by
$$g(\a) = \begin{cases}
    &\sum_{i=1}^k \lceil \frac{a_i - 1}{3} \rceil \text{ if } \alpha_1 = 0 \text{ and } \alpha_2 \neq 0 \\
    &1 + \sum_{i=1}^k \lceil \frac{a_i - 1}{3} \rceil \text{ otherwise}.
\end{cases}.$$
We may further assume that $a_j = 2 \pmod 3$ for $j = 1, \ldots, \alpha_2$, $a_j = 0 \pmod 3$ for $j = \alpha_2+1, \ldots, \alpha_2 + \alpha_0$ and $a_j = 1 \pmod 3$ for $j = \alpha_0 + \alpha_2+1, \ldots, k$. Let $\beta_1 = \min (\alpha_2,t-1)$, $\beta_2 = \min(\alpha_0, \lfloor \frac{\max(0,t-1-\alpha_2)}{2} \rfloor)$ and $\beta_3 = \lfloor \frac{\max (0, t-1 - \beta_1 - 2 \beta_2)}{3} \rfloor$. We define $\b \in \NN^k$ as follows. 
$$b_i = \begin{cases}
 a_i - 1 & \text{ for }  i = 1, \ldots, \beta_1,\\
 a_i -2 & \text{ for } i = \alpha_2+1, \ldots, \alpha_2 + \beta_2.
 \end{cases}$$
With this notation, we have: 
\begin{thm}\label{depth_star_like_trees}
Let $\a = (a_1, \ldots, a_k) \in \NN^k$ be a tuple of positive integers. Denote $T_\a$ the starlike trees obtained by joining paths of length $a_1, \ldots, a_k$ at the root $1$. Then for all $t$ such that $1 \le t \le |\a| - k = s$,
$$        \depth S/I(T_\a)^t = g(\b)- \beta_3.$$
\end{thm}
\begin{exm}Let $\a = (3,3,5)$. Then $\alpha_0 = 2$, $\alpha_1 = 0$ and $\alpha_2 = 1$. By Theorem \ref{depth_star_like_trees}, we see that the sequence $\{\depth (S/I(T_\a)^t) \mid 1 \le t \le 9\}$ is $\{4,4,4,3,3,2,2,2,1\}$.     
\end{exm}

Thus, in general, one cannot expect a simple formula for the depth of powers of the edge ideal of a tree in terms of its combinatorial invariants. We also note that starlike trees are natural generalizations of paths, as the join of two paths of length $a_1$ and $a_2$ is a path of length $a_1 + a_2 + 1$. 

We structure the paper as follows. In Section \ref{sec_pre}, we set up the notation and provide some background. In Section \ref{sec_depth_power}, we prove Theorem \ref{depth_power_cycle}. In Section \ref{sec_depth_power_tree}, we prove Theorem \ref{depth_star_like_trees}.

\section{Preliminaries}\label{sec_pre}
In this section, we recall some definitions and properties concerning the depth of monomial ideals and edge ideals of graphs. The interested readers are referred to \cite{BH, D} for more details.

Throughout the paper, we denote $S = \k[x_1,\ldots, x_n]$ a standard graded polynomial ring over a field $\k$. Let $\m = (x_1,\ldots, x_n)$ be the maximal homogeneous ideal of $S$.

\subsection{Depth} For a finitely generated graded $S$-module $L$, the depth of $L$ is defined to be
$$\depth(L) = \min\{i \mid H_{\m}^i(L) \ne 0\},$$
where $H^{i}_{\m}(L)$ denotes the $i$-th local cohomology module of $L$ with respect to $\m$. We have the following estimates on depth along a short exact sequence (see \cite[Proposition 1.2.9]{BH}).

\begin{lem}\label{lem_short_exact_seq}
    Let $0 \to L \to M \to N \to 0$ be a short exact sequence of finitely generated graded $S$-modules. Then 
    \begin{enumerate}
        \item $\depth M \ge \min (\depth L, \depth N),$
        \item $\depth L \ge \min (\depth M, \depth N + 1).$
    \end{enumerate}
\end{lem}

We make repeated use of the following two results in the sequence. The first one is \cite[Corollary 1.3]{R}. The second one is \cite[Theorem 4.3]{CHHKTT}.

\begin{lem}\label{lem_upperbound} Let $I$ be a monomial ideal and $f$ a monomial such that $f \notin I$. Then
    $$\depth S/I \le \depth S/(I:f)$$
\end{lem}

\begin{lem}\label{lem_depth_colon} Let $I$ be a monomial ideal and $f$ a monomial. Then 
$$\depth S/I \in \{\depth (S/I:f), \depth (S/(I,f))\}.$$    
\end{lem}

We also use the following result about the depth of powers of sums of ideals \cite[Theorem 1.1]{NV2}.
\begin{thm}\label{thm_mixed_sum} Let $I$ and $J$ be monomial ideals in standard graded polynomial rings $R$ and $S$ over a field $\k$. Denote $P = I +J$ the sum of ideals in $T = R \otimes_\k S$. Then $\depth T/P^s$ equals 
$$\min_{i\in [1,s-1], j \in [1,s]} \{ \depth R/I^{s-i} + \depth S/J^{i} + 1, \depth R/I^{s-j+1} + \depth S/J^{j}\}.$$    
\end{thm}

\subsection{Graphs and their edge ideals} 

Let $G$ denote a finite simple graph over the vertex set $V(G)=[n] = \{1,2,\ldots,n\}$ and the edge set $E(G)$. For a vertex $x\in V(G)$, let the neighbours of $x$ be the subset $N_G(x)=\{y\in V(G) \mid \{x,y\}\in E(G)\}$, and set $N_G[x]=N_G(x)\cup\{x\}$. 

The edge ideal of $G$ is defined to be
$$I(G)=(x_ix_j~|~\{i,j\}\in E(G))\subseteq S.$$
For simplicity, we often write $i \in G$ (resp. $ij \in G$) instead of $i \in V(G)$ (resp. $\{i,j\} \in E(G)$). By abuse of notation, we also call $x_i x_j \in I(G)$ an edge of $G$.

We now recall several classes of graphs that we study in this work. A tree is a connected acyclic graph. A path $P_n$ of length $n -1$ is the graph on $[n]$ whose edges are $\{i,i+1\}$ for $i = 1, \ldots,n-1$. A starlike tree consists of several paths attached to a root.

A cycle $C_n$ of length $n \ge 3$ is the graph on $[n]$ whose edges are $\{i,i+1\}$ for $i = 1, \ldots, n-1$ and $\{1,n\}$. A graph $G$ is called weakly chordal if $G$ and its complement do not contain an induced cycle of length $\ge 5$.

We will use the following result \cite[Lemma 2.10]{Mo} later.

\begin{lem}\label{lem_colon_path} Suppose $G$ is a graph, $x$ is a leaf of $G$ and $y$ is the unique neighbour of $x$. Then for any $t \ge 2$,
$$I(G)^t : (xy) = I(G)^{t-1}.$$
\end{lem}

Furthermore, the depth of $I(P_n)$ and $I(C_n)$ are well known \cite{Mo, BC2}.
\begin{lem}\label{lem_depth_path_cycle} We have 
\begin{enumerate}
    \item $\depth S/I(P_n) = \lceil \frac{n}{3} \rceil$ for $n \ge 2$,
    \item $\depth S/I(C_n) = \lceil \frac{n-1}{3} \rceil$ for $n \ge 3$.
\end{enumerate}    
\end{lem}

\subsection{Even-connection and colon of edge ideals} Let $I = I(G)$ be the edge ideal of a graph $G$. When $f = e_1 \cdots e_{t-1}$ is a product of edges of $G$, Banerjee \cite{B} showed that the colon ideal $I^t:f$ is generated by quadratic monomials $uv$ such that $u$ and $v$ are even-connected via $f$. We now describe this result and an application that will be useful later.

\begin{defn} We say that $u$ and $v$ are $f$-even connected if there exists vertices $x_1, \ldots, x_{2k}$ such that $ux_1, x_ix_{i+1}$ and $x_{2k}v$ are edges of $G$ and $x_1x_2, x_3x_4, \ldots, x_{2k-1}x_{2k}$ are among $\{e_1, \ldots, e_{t-1} \}$ and $\prod_{j=1}^{2k} x_j$ divides $f$.
\end{defn}
The following is \cite[Theorem 6.7]{B}.
\begin{thm}\label{thm_even_connection} Let $I = I(G)$ be the edge ideal of a graph $G$ and $f = e_1 \cdots e_{t-1}$ be a product of edges of $G$. Then $I^t:f$ is generated by quadratic monomials $uv$ such that $u$ and $v$ are $f$-even connected.   
\end{thm}

Using this, we prove 
\begin{lem}\label{lem_colon_product_edges} Let $f = e_1 \cdots e_{t}$ where $e_1, \ldots, e_{t}$ are distinct edges of $G$. Assume that the induced subgraph of $G$ on $N[\supp f] = \cup_{i=1}^{t} N[\supp e_i]$ does not contain an odd cycle. For each $s$, denote $J_s = I^s:(e_1 \cdots e_{s-1})$. Then 
$$J_{t+1} = J_t : u \cap J_t:v$$
where $e_t = uv$.    
\end{lem}
\begin{proof} First, we prove that the left hand side is contained in the right hand side. Assume that $wz \in J_{t+1}$. If $wz\in I$, then clearly $wz \in J_t \subseteq J_t : u \cap J_t : v$. Thus, we may assume that $wz \notin I$. By Theorem \ref{thm_even_connection} there exists $x_1, \ldots,x_{2k}$ such that $wx_1, x_ix_{i+1},x_{2k}z$ are edges of $G$ and $x_{2i+1}x_{2i+2}$ are among $e_1, \ldots, e_t$. If $e_t$ does not appears among the edges $\{x_1x_2, x_3x_4, \ldots, x_{2k-1}x_{2k}\}$ then $wz \in J_t \subseteq J_t : u \cap J_t:v$. Thus, we may assume that $x_{2i-1} = u$ and $x_{2i} = v$. Hence, by Theorem \ref{thm_even_connection}, $wu$ and $vz$ are in $J_t$. Hence $wz \in J_t : u \cap J_t : v$.

It remains to prove that $g \in J_t:u \cap J_t:v$ then $g \in J_{t+1}$. First, note that $J_t : u = J_t + (U)$ and $J_t:v = J_s + (V)$ where $U$ and $V$ consisting of variables that are even-connected to $u$ and $v$ respectively. Hence, $f \in J_s + (U) \cap (V)$. First, we prove that $U \cap V = \emptyset$. Indeed, if $w \in U \cap V$, then $w$ is even-connected to $u$ and $v$. Hence, the walks from $w$ to $u$ and $w$ to $v$ form an closed odd walks, a contradiction to the assumption. Now, if $f \in J_t$ then clearly $f \in J_{t+1}$. Thus, we may assume that $f = wz$ with $w \in U$ and $z \in V$. If either $wu$ or $zv$ belongs to $I$, then it is clear that $wz \in J_{s+1}$. Thus, we may assume that there exists even walks $wx_1 \ldots x_{2k}u$ and $vy_1, \ldots, y_{2l}z$. If $\{x_1, \ldots, x_{2k} \} \cap \{y_1, \ldots, y_{2l} \} = \emptyset$ then $wz$ are $f$-evenly connected by definition. If $x_{2i} = y_{2j+1}$ then the walk $x_{2i} \ldots x_{2k}uvy_1 \ldots y_{2j+1}$ is a closed odd walk. Hence, the restriction of $G$ on $N[\supp f]$ contains an odd cycle, a contradiction. Thus, we must have $x_{2i} = y_{2j}$ for some $i$ and $j$. Now along the even walk $wx_1\ldots x_{2k} u v y_1 \ldots y_{2l} z$ if $x_{2i} = y_{2j}$ then we shorten the walk by omitting the middle path from $x_{2i}$ to $y_{2j}$ and still have an even walk. We can repeat this until there is no further repetition of the vertices on the walk. The conclusion follows.    
\end{proof}

\subsection{Projective dimension of edge ideals of weakly chordal graphs}
We note that the colon ideals of powers of edge ideals of trees by products of edges are edge ideals of weakly chordal graphs. Their projective dimension can be computed via the notion of strongly disjoint families of complete bipartite subgraphs, introduced by Kimura \cite{K}. For a graph G, we consider all families of (non-induced) subgraphs
$B_1, \ldots, B_g$ of $G$ such that
\begin{enumerate}
    \item each $B_i$ is a complete bipartite graph for $1 \le i \le g$,
    \item the graphs $B_1, \ldots, B_g$ have pairwise disjoint vertex sets,
    \item there exist an induced matching $e_1,\ldots, e_g$ of $G$ for each $e_i \in E(B_i)$ for $1\le i \le g$.
\end{enumerate}
Such a family is termed a strongly disjoint family of complete bipartite subgraphs. We define 
$$ d(G) = \max ( \sum_1^g |V (B_i)| - g),$$
where the maximum is taken over all the strongly disjoint families of complete bipartite subgraphs $B_1, \ldots, B_g$ of $G$. We have the following \cite[Theorem 7.7]{NV1}.

\begin{thm}\label{thm_pd_weakly_chordal}
    Let $G$ be a weakly chordal graph with at least one edge. Then 
    $$\pd (S/I(G)) = d(G).$$
\end{thm}

\section{Depth functions of powers of cycles}\label{sec_depth_power}
In this section, we compute the depth of powers of paths and cycles. For a real number $a$, denote $\lceil a \rceil$ the least integer at least $a$, $\lfloor a \rfloor$ the largest integer at most $a$. First, we have a simple lemma.

\begin{lem}\label{lem_sum_fraction} Let $a,b$ be integers. Then 
$$\lceil \frac{a}{3} \rceil + \lceil \frac{b}{3} \rceil \ge \lceil \frac{a+b}{3} \rceil.$$    
\end{lem}
\begin{proof}
    There exist unique integers $k$, $a_1$, $l, b_1$ such that $a = 3k + a_1$ and $b = 3l + b_1$ with $1 \le a_1,b_1 \le 3$. By definition, $\lceil a/3 \rceil = k + 1$ and $\lceil b/3  \rceil = l + 1$. Since $a_1, b_1 \le 3$, $a_1 + b_1 \le 6$, hence $2 \ge \lceil (a_1 + b_1)/3 \rceil$. The conclusion follows.
\end{proof}

We fix the following notation throughout the rest of the paper. For each $n$, $P_n$ denotes a path of length $n-1$ and $C_n$ denotes a cycle of length $n$. We denote $e_i = x_i x_{i+1}$ for $i = 1, \ldots, n-1$ and $e_n = x_1 x_{n}$. We define $\varphi(n,t) = \lceil \frac{n-t+1}{3} \rceil$. 

The following lemma is a crucial step in computing the depth of powers of edge ideals of paths and cycles.

\begin{lem}\label{lem_lowerbound_path} Let $H$ be any subgraph of $P_n$. Then for any $t < n$
$$ \depth (S/ (I(P_n)^t + I(H))) \ge \varphi(n,t).$$    
\end{lem}
\begin{proof}
    We will prove the statement by downward induction on the size of $H$ and induction on $t$. If $|H| = n-1$ or $t = 1$, then $I(P_n)^t + I(H) = I(P_n)$. By Lemma \ref{lem_depth_path_cycle}, we have $\depth S/I(P_n) = \lceil n/3 \rceil \ge \varphi(n,t)$. Furthermore, $\sqrt{I(P_n)^t + I(H)} = I(P_n)$. In particular, $\m$ is not an associated prime of $I(P_n)^t + I(H)$ for any $t$ and $H$. Hence, if $n \le 4$ and $t \ge 2$ then $\depth (S/ (I(P_n)^t + I(H))) \ge 1 \ge \varphi(n,t)$. Thus, we may assume that $n \ge 5$ and $t \ge 2$.

    For each $j = 1, \ldots, n-1$, denote $e_j$ the edge $x_jx_{j+1}$. Now assume that $i$ is the smallest index such that $e_i \notin H$, i.e., $e_1, \ldots, e_{i-1} \in H$. Let $J = I(P_n)^t + I(H) = I(P_{n-i+1})^t + I(H)$, where $P_{n-i+1}$ is the path consisting of the edges $e_i, \ldots, e_{n-1}$. By Lemma \ref{lem_depth_colon},
    $$\depth (S/J)  \in \{ \depth (S/J:e_i), \depth (S/(J,e_i) \}$$
    By induction, it suffices to prove that 
    \begin{equation}\label{eq_3_1}
      \depth (S/J:e_i) \ge \varphi(n,t).  
    \end{equation}
    By Lemma \ref{lem_colon_path}, 
    \begin{equation}\label{eq_3_2}
        J:e_i = I(P_{n-i+1})^{t-1} + (I(H) : e_i).
    \end{equation}
    
    There are two cases as follows.

\noindent \textbf{Case 1.} $i = 1$. There are two subcases: 

    \textbf{Case 1.a.} $e_2 \in H$. Equation \eqref{eq_3_2} becomes 
    \begin{equation}\label{eq_3_3}
        J:e_1 = (e_1 + P_{n-3})^{t-1} + (x_3) + I(H'),
    \end{equation} where $P_{n-3}$ is the path from $x_4$ to $x_n$ and $H'$ is the induced subgraph of $H$ on $\supp H \setminus [3]$. For each $\ell \ge 1$, denote $K_\ell = (e_1 + P_{n-3})^\ell + (x_3) + I(H')$. In particular, $J:e_1 = K_{t-1}$. Since $x_1,x_2$ does not appears in $P_{n-3}$ nor $I(H')$. We have,
    \begin{equation}\label{eq_3_4}
        K_{\ell+1} : e_1 = K_{\ell}.
    \end{equation}
    Applying Lemma \ref{lem_depth_colon} with $f = e_1$, we have 
    \begin{equation}\label{eq_3_5}
        \depth (S/K_{\ell + 1}) \in \{ \depth (S/K_{\ell}),\depth (S/(K_{\ell+1} + e_1)).
    \end{equation}
    Hence,
    \begin{equation}\label{eq_3_6}
        \depth (S/J:e_1) = \depth (S/K_{t-1}) \in \{ \depth (S/(K_\ell + e_1)) \mid \ell = 1, \ldots, t-1\}.
    \end{equation}
    Since $K_\ell + e_1 = P_{n-3}^\ell + I(H') + e_1 + (x_3)$, to establish \eqref{eq_3_1} it remains to prove that 
    $$\depth S/( P_{n-3}^\ell + I(H')  + e_1 + (x_3)) \ge \varphi(n,t),$$
    for all $\ell = 1, \ldots, t-1$. Since $\supp P_{n-3} \cap \supp e_1 = \emptyset$, by Theorem \ref{thm_mixed_sum}
    $$\depth S/( P_{n-3}^{\ell} + I(H')  + e_1 + (x_3)) = \depth R/(P_{n-3}^{t-1} + I(H')) + 1,$$
    where $R = \k[x_4, \ldots,x_n]$. By induction, we deduce that 
    $$\depth S/( P_{n-3}^{\ell} + I(H')  + e_1 + (x_3)) \ge \varphi(n-3,\ell) + 1 = \lceil \frac{n-\ell + 1}{3} \rceil \ge \varphi(n,t),$$
    for all $\ell = 1, \ldots, t-1$. The conclusion follows.

    \textbf{Case 1.b.} $e_2 \notin H$. Equation \eqref{eq_3_2} becomes $J:e_1 = I(P_n)^{t-1} + I(H)$. The conclusion follows from induction on $t$.

\vspace{2mm}

\noindent \textbf{Case 2.} $i > 1$. By definition $e_1, \ldots,e_{i-1} \in H$. There are two cases as before: 

    \textbf{Case 2.a.} $e_{i+1} \in H$. Equation \eqref{eq_3_2} becomes
    \begin{equation}
        J:e_i = (e_i + P_{n-i-2})^{t-1} + (x_{i-1}, x_{i+2}) + I(P_{i-2}) + I(H'),
    \end{equation}
    with $P_{n-i-2}$ is the path $i+3, \ldots, n$, $P_{i-2}$ is the path $1, \ldots, i-2$ and $H'$ is the induced subgraph of $H$ on $\supp H \setminus \{1, \ldots, i-1, i + 2\}.$ With an argument similar to Case 1.a, by repeated use of Lemma \ref{lem_depth_colon} it suffices to prove that
    $$\depth S/(P_{n-i-2}^{\ell} + I(H') + e_i + I(P_{i-2}) + (x_{i-1},x_{i+2})) \ge \varphi(n,t),$$
    for all $\ell = 1, \ldots, t-1$. Note that $P_{i-2}$, $e_i$, and $P_{n-i-2}$ have support on a different set of variables. By induction, Lemma \ref{lem_depth_path_cycle}, and Theorem \ref{thm_mixed_sum}, the left hand side is at least 
    $$\varphi(n-i-2,\ell) + 1 + \lceil \frac{i-2}{3} \rceil  = \lceil \frac{n-i-\ell - 1}{3} \rceil + 1 + \lceil \frac{i-2}{3} \rceil \ge \varphi(n,t).$$
    
    \textbf{Case 2.b.} $e_{i+1} \notin H$. Equation \eqref{eq_3_2} becomes \begin{equation}
        J:e_i = I(P_{n-i+1})^{t-1} + (x_{i-1}) + I(P_{i-2}) + I(H'),
    \end{equation}
    with $P_{n-i+1}$ is the path $i, \ldots, n$, $P_{i-2}$ is the path $1, \ldots, i-2$ and $H'$ is the induced subgraph of $H$ on $\supp H \setminus \{1, \ldots, i-1\}.$ By induction and Theorem \ref{thm_mixed_sum}, we have 
    $$\depth (S/J:e_i) \ge \varphi(n-i+1,t-1)  + \lceil \frac{i-2}{3} \rceil = \lceil \frac{n-i - t + 3}{3} \rceil + \lceil \frac{i-2}{3} \rceil \ge \varphi(n,t).$$
    The conclusion follows.
\end{proof}

To obtain an upper bound for $\depth S/I(P_n)^t$, we prove
\begin{lem}\label{lem_upperbound_path}
    Let $I = (x_1x_2,\ldots, x_{n-1}x_n)$ be the edge ideal of a path of length $n-1$ on $n$ vertices. Denote $e_i = x_ix_{i+1}$ for all $i = 1, \ldots, n-1$. Then for all $t \le n -2$, we have 
    $$\depth (S/(I^t:(e_2\ldots e_t))) = \lceil \frac{n -t + 1}{3} \rceil.$$
\end{lem}
\begin{proof} By Lemma \ref{lem_colon_product_edges}, we see that $I^t : (e_2 \ldots e_t)$ is the edge ideal of the following graph 
$$G_{n,t} = P_n \cup \{x_i x_j \mid i < j  \le t + 2\text{ is of different parity}\}.$$
We will prove by induction on $n$ and $n-t$ that 
$$\depth (S/I(G_{n,t})) = \lceil \frac{n-t+1}{3} \rceil.$$
If $t = n-2$, then $G_{n,t}$ is a complete bipartite graph, hence $\depth (S/I(G_{n,t})) = 1$. Thus, we may assume that $t \le n-3$. Hence, 
$I(G_{n,t}) = I(G_{n-1,t}) + e_{n-1}$. Furthermore, this decomposition is a Betti splitting by \cite[Corollary 4.12]{NV1}. Since $I(G_{n-1,t}) \cap e_{n-1} = e_{n-1} (x_{n-2} + I(G_{n-3,t}))$, by \cite[Corollary 4.8]{NV1} and induction we have
\begin{align*}
    \pd (S/I(G_{n,t})) &= \max (\pd (S/I(G_{n-1,t})), 1, \pd (S/I(G_{n-3,t})) + 1) \\
    &= \max (n-1 - \varphi(n-1,t),1,n- \varphi(n-3,t) - 1) = n-\varphi(n,t).
\end{align*}
The conclusion follows from the Auslander-Buchsbaum formula.
\end{proof}

\begin{thm}\label{thm_depth_path} Let $I(P_n)$ be the edge ideal of a path of length $n-1$. Then 
$$\depth (S/I^t) = \max( \lceil \frac{n -t + 1}{3} \rceil,1),$$
    for all $t \ge 1$.
\end{thm}
\begin{proof} By Lemma \ref{lem_depth_path_cycle} and \cite{T}, we may assume that $2 \le t \le n-3$. By Lemma \ref{lem_lowerbound_path}, take $H$ be the empty graph, we deduce that $\depth S/I^t \ge \varphi(n,t)$. The conclusion then follows from Lemma \ref{lem_upperbound_path} and Lemma \ref{lem_upperbound}.    
\end{proof}

\begin{rem} Note that for any integer $n$, $\lceil \frac{n}{3} \rceil = n + 1 - \lfloor \frac{n+1}{3} \rfloor - \lceil \frac{n+1}{3} \rceil$. In particular, Theorem \ref{thm_depth_path} is a special case of \cite[Theorem 1]{BC1}. We include a simple argument here because Lemma \ref{lem_lowerbound_path} will be critical to deduce the formula for depth of powers of edge ideals of cycles.    
\end{rem}

We will now turn to the edge ideals of cycles $C_n$. The depth of powers of $I(C_n)$ in the case $n = 3$ and $n = 4$ are clear. Thus, we may assume that $n \ge 5$. By \cite{T}, we know that $\dstab(I(C_n)) = \lceil \frac{n+1}{2} \rceil.$ Thus, we may assume that $2 \le t < \lceil \frac{n+1}{2} \rceil.$ First, we note that $f = x_1 \cdots x_{2t-2}$ is a product of distinct variables. By Lemma \ref{lem_depth_colon}, we have 
\begin{equation}\label{eq_3_8}
    \depth S/I(C_n)^t \in \{ \depth (S/ I(C_n)^t : f), \depth (S/ (I(C_n)^t,f)) \}.
\end{equation}
Hence, to establish the lower bound for $\depth (S/ I(C_n)^t)$, it suffices to prove that $\depth S/(I^t:f) \ge \varphi(n,t)$ and $\depth (S/(I^t,f)) \ge \varphi(n,t)$. We establish the first inequality in the following

\begin{lem}\label{lem_lower_bound_cycle_1} Assume that $n \ge 5$ and $2 \le t < \lceil \frac{n+1}{2} \rceil$. Then 
    $$\depth (S/I(C_n)^t : (x_1 \cdots x_{2t-2})) \ge \varphi(n,t).$$
\end{lem}
\begin{proof} For each $t = 1, \ldots, \lceil \frac{n+1}{2} \rceil - 1$, let $J_t = I^t : (x_1 \cdots x_{2t-2})$. By Lemma \ref{lem_colon_product_edges}, 
\begin{equation}\label{eq_3_9}
    J_{t+1} = (J_t : x_{2t-1} ) \cap (J_t : x_{2t}).
\end{equation}
Note that $\depth (S/J_1) = \depth S/I(C_n) = \lceil \frac{n-1}{3} \rceil = \varphi(n,2)$. First, consider the base case $t = 2$. Note that $I:x_1 + I:x_2 = (x_1,x_2,x_3,x_n) + I(P_{n-4})$, where $P_{n-4}$ is the path from $4$ to $n-1$. Hence, 
\begin{equation}
    \depth (S/ (I:x_1 + I:x_2)) = \lceil \frac{n-4}{3} \rceil = \varphi(n,2) - 1.
\end{equation}
By Lemma \ref{lem_short_exact_seq},
\begin{equation}
    \depth (S/ J_2 )  \ge \min \{ \depth S/ J_1, \depth (S/(I:x_1 + I:x_2) + 1) \} = \varphi(n,2).
\end{equation}
Now, consider the induction step. We have 
    \begin{equation}
        J_t : x_{2t-1} + J_t : x_{2t} = (x_n,x_2,x_4, \ldots, x_{2t-4},x_{2t-2},x_{2t-1},x_{2t},x_{2t+1}) + I(P_{n-2t-2})
    \end{equation}
    where $P_{n-2t-2}$ is the path from $2t+2$ to $n-1$. Note that $x_1, x_3, \ldots, x_{2t-3}$ are variables that do not appear in $J_t:x_{2t-1} + J_t:x_{2t}$. Hence,
    \begin{equation}
        \depth (S/ (J_t : x_{2t-1}) + (J_t : x_{2t}) ) = t-1 + \lceil \frac{n-2t-2}{3} \rceil \ge \varphi (n,t+1) - 1.
    \end{equation}
Together with equation \eqref{eq_3_9}, Lemma \ref{lem_short_exact_seq}, Lemma \ref{lem_upperbound}, and induction, we have
    \begin{align*}
        \depth (S/J_{t+1}) &\ge \min \{ \varphi(n,t), \depth S/ (J_t : x_{2t-1} + J_t : x_{2t}) + 1\}\\
        &\ge \varphi(n,t+1).
    \end{align*}
The conclusion follows.
\end{proof}

The second inequality is established in the following 
\begin{lem} \label{lem_lower_bound_cycle_2} Assume that $t \ge 2$ and $f = x_1 \cdots x_{2t-2}$. Then $\depth (S/ (I^t,f)) \ge \varphi(n,t)$.    
\end{lem}
\begin{proof}
    For each $j = 1, \ldots, t-2$, denote $f_j = e_{2j-1} \cdots e_{2t-3}$. Then $f = f_1$ and $f_j = e_{2j-1} f_{j+1}$. In other words, for any subgraph $H$ of $G$ consisting of edges which are subsets of $\{e_1, e_3, \ldots, e_{2j-3}\}$, we have
\begin{equation}
I^t + I(H) + (f_j) = (I^t + I(H) + (e_{2j-1})) \cap (I^t + I(H) + (f_{j+1})).  
\end{equation}
The conclusion follows from Lemma \ref{lem_short_exact_seq} and the following Lemma.
\end{proof}

\begin{lem}\label{lem_lowerbound_cycle_aux} Let $H$ be a non-empty subgraph of $G$. Then for $t \ge 2$,
    $$\depth (S/(I(C_n)^t + I(H))) \ge \varphi(n,t).$$
\end{lem}
\begin{proof} Since $H$ is non-empty, we may assume that $e_n = x_1x_n \in H$. We now proceed similarly as in Lemma \ref{lem_lowerbound_path}. We will prove by downward induction on the size of $H$. If $|H| = n$, then $I(C_n)^t + I(H) = I(C_n)$. Hence 
$$\depth (S/I(C_n)^t + I(H)) = \varphi(n,2) \ge \varphi(n,t).$$
Let $i$ be the smallest index such that $e_i \notin H$, i.e., $e_0 = e_n,e_1,\ldots, e_{i-1} \in H$. Let $J = I(C_n)^t + I(H)$. By Lemma \ref{lem_depth_colon}, we have 
$$\depth (S/J \in \{ \depth (S/ J: e_i, \depth S/(J,e_i))\}.$$ 
By induction, it suffices to prove that 
\begin{equation}
    \depth (S/ J : e_i) \ge \varphi(n,t).
\end{equation}
Note that $e_{i-1} \in H$. Hence $x_{i-1} \in I(H) : e_i$. There are two cases. 

\noindent \textbf{Case 1.} $e_{i+1} \notin H$. By Lemma \ref{lem_colon_path}, we have
\begin{equation}
    J : e_i = (x_{i-1}) + I(P_{n-1})^{t-1} + I (H'),
\end{equation}
where $P_{n-1}$ is the path on $n-1$ vertices $[n] \setminus \{i-1\}$ and $H'$ is the induced subgraph of $H$ on $\supp H \setminus \{i-1\}$. By Lemma \ref{lem_lowerbound_path}, we deduce that 
$$\depth (S/J : e_i) = \depth (R/I(P_{n-1}))^{t-1} + I(H')) \ge \varphi(n-1,t-1) = \varphi(n,t),$$
where $R = K[x_1,\ldots,\widehat {x_{i-1}}, x_i, \ldots,x_n]$.

\vspace{2mm}

\noindent \textbf{Case 2.} $e_{i+1} \in H$. By Lemma \ref{lem_colon_path}, we have 
\begin{equation}
    J : e_i = (x_{i-1},x_{i+2}) + (e_i,I(P_{n-4}))^{t-1} + I(H'),
\end{equation}
where $P_{n-4}$ is the path on $[n] \setminus \{i-1,i,i+1,i+2\}$ and $H'$ is the induced subgraph of $H$ on $\supp H \setminus \{i-1,i,i+1,i+2\}$. For each $\ell = 1, \ldots, t-1$, let 
\begin{equation}
    K_\ell = (x_{i-1},x_{i+2}) + (e_i,I(P_{n-4}))^\ell + I(H').
\end{equation}
Then $K_{\ell+1} : e_i = K_{\ell}.$ By Lemma \ref{lem_depth_colon}, it suffices to prove that 
\begin{equation}
    \depth S/(K_\ell+e_i) \ge \varphi(n,t),
\end{equation}
for all $\ell = 1, \ldots, t-1$. By Theorem \ref{thm_mixed_sum} and Lemma \ref{lem_lowerbound_path}, 
\begin{equation}
    \depth S/(K_\ell + e_i) = 1 + \depth R/ (I(P_{n-4})^\ell + I(H')) \ge 1 + \varphi(n-4,\ell) \ge \varphi (n,t),
\end{equation}
for all $\ell = 1, \ldots,t-1$, where $R = \k[x_1, \ldots, x_{i-2}, x_{i+3}, \ldots, x_n]$. 

The conclusion follows.
\end{proof}

We now give an upper bound for the depth of powers of cycles.

\begin{lem}\label{lem_upperbound_cycle} Assume that $I = I(C_n)$ and $t \le n-2$. Then 
$$\depth (S/I^{t} : (e_2 \cdots e_t)) \le \lceil \frac{n-t+1}{3} \rceil.$$
\end{lem}
\begin{proof} Let $J = I^t : (e_2 \cdots e_t)$. By Lemma \ref{lem_colon_product_edges}, we see that $J$ is the edge ideal of the following graph 
$$G_{n,t} = C_n \cup \{ x_i x_j \mid i < j \le t+2 \text{ is of different parity} \}.$$    
First, assume that $t = n-2$. If $n$ is even, then $G_{n,t}$ is a complete bipartite graph, hence $\depth S/J = 1$. 
If $n$ is odd, let $H$ be the restriction of $G_{n,t}$ to $[n] \setminus \{1\}$. Then $H$ is a complete bipartite graph. Furthermore, we have $J = x_1 (x_2, x_4, \ldots, x_{n-1},x_n) + I(H)$. By \cite[Corollary 4.12]{NV1}, this is a Betti splitting. Furthermore, $x_1 (x_2, x_4 \ldots, x_{n-1},x_n) \cap I(H) = x_1 I(H)$. Hence, by \cite[Corollary 4.8]{NV1},
$$\pd (S/J) = \pd (S/x_1I(H)) + 1 = n-1.$$

If $t = n -3$ then $J:x_{n} = (x_1,x_{n-1}) + \text{complete bipartite graph}$, hence $\depth S/J:x_n \le 2$. Now assume that $t \le n - 4$. Then 
\begin{equation}
    J : x_{n-1} = (x_n,x_{n-2}) + I(P_{n-3})^t : (e_2 \cdots e_t). 
\end{equation}
    By Lemma \ref{lem_upperbound_path}, $\depth (S/J:x_{n-1}) = 1 + \varphi(n-3,t) = \varphi(n,t).$ By Lemma \ref{lem_upperbound}, the conclusion follows.
\end{proof}

We are now ready for the main result of this section.

\begin{thm}\label{thm_depth_powers_cycles}
    Let $C_n$ be a cycle on $n \ge 5$ vertices. Then for $2 \le t < \lceil \frac{n+1}{2} \rceil$,
    $$\depth (S/I(C_n)^t) = \varphi(n,t).$$
\end{thm}
\begin{proof}
    By Lemma \ref{lem_depth_colon}, Lemma \ref{lem_lower_bound_cycle_1} and Lemma \ref{lem_lower_bound_cycle_2}, we get $\depth (S/I^t) \ge \varphi(n,t)$. By Lemma \ref{lem_upperbound} and Lemma  \ref{lem_upperbound_cycle}, we get $\depth S/I^t \le \varphi(n,t)$. The conclusion follows.
\end{proof}

\begin{rem}
    Our arguments extend to compute the depth of symbolic powers of cycles. We will cover that in subsequent work \cite{MTV}
\end{rem}

\section{Depth of powers of edge ideals of trees}\label{sec_depth_power_tree}
In this section, we study the depth of powers of edge ideals of trees. First, we have 

\begin{thm}\label{thm_depth_tree} Let $T$ be a tree. Then $\depth S/I(T) = q(T)$, where $q(T)$ is the minimum size of a maximal independent set of $T$.    
\end{thm}
\begin{proof} This follows from work of Kimura \cite{K}. We give an alternative proof here. 

Let $G$ be any simple graph. Let $F$ be any maximal independent set of $G$. Then $I(G) : x_F = (x_j \mid j \in [n] \setminus F)$. Hence, $\depth S/(I(G) : x_F) = |F|$. By Lemma \ref{lem_upperbound}, $\depth S/I(G) \le q(G)$, where $q(G)$ is the minimum size of a maximal independent set of $G$. 

It remains to prove that $\depth S/I(T) \ge q(T)$. Let $F$ be a maximal independent set of minimum size. We first claim that we can take $F$ so that $F$ contains a near leaf. Indeed, if all vertices of $F$ are leaves, we may replace a leaf with its unique neighbor $w$ and remove all leaves in $F \cap N(w)$. Then we have another maximal independent set of size at most $F$. Thus, we may assume that $w$ is a near leaf. Write $F = w \cup F_1$. By Lemma \ref{lem_depth_colon}, we have $\depth S/I \in \{ \depth S/I:w, \depth S/(I,w) \}$. Let $T_1$ be the tree obtained by removing $w$ and all of its neighbors that are leaves. Then $\depth S/(I,w) = |L(w)| + \depth R/I(T_1)$ where $L(w)$ is the set of leaves attached to $w$. By induction, $\depth R/I(T_1) \ge q(T_1)$. Furthermore, $q(T) \le q(T_1) + 1$. Hence, $\depth S/(I,w) \ge q(T)$. 

Finally, $I:w = (x_j \mid j \in N(w)) + I(T_2)$ where $T_2$ is the tree obtained by removing $N[w]$. We have $\depth S/I:w = 1 + \depth R/I(T_2)$ and $q(T) \le q(T_2) + 1$. The conclusion follows.     
\end{proof}

In particular, we have the following interesting property on the depth of trees with respect to inclusion.
\begin{cor}\label{cor_depth_increase} Let $T_1$ be a subtree of $T_2$. Then 
$$\depth R/I(T_1) \le \depth S/I(T_2),$$
where $R$ and $S$ are polynomial rings over the variables corresponding to supports of $T_1$ and $T_2$.
\end{cor}

We will now compute the depth of powers of edge ideals of starlike trees. First, we introduce some notation and a function. We use bold letters for vectors in $\NN^k$. For a vector $\a \in \NN^k$, its components are $a_1, \ldots, a_k$ and $|\a| = a_1 + \ldots + a_k$. The vectors $\e_1, \ldots, \e_k$ are the canonical unit vectors of $\RR^k$; $\mathbf{1}$ denotes the vector whose all components are $1$. For a tuple $\a \in \NN^k$, we let $\alpha_j$ be the number of $a_i$ such that $a_i = j \pmod 3$. Let $g: \NN^k \to \NN$ be defined by 
$$g(\a) = \begin{cases}
    &\sum_{i=1}^k \lceil \frac{a_i - 1}{3} \rceil \text{ if } \alpha_1 = 0 \text{ and } \alpha_2 \neq 0 \\
    &1 + \sum_{i=1}^k \lceil \frac{a_i - 1}{3} \rceil \text{ otherwise}.
\end{cases}$$
The following properties of $g$ follow immediately from the definition.
\begin{lem}\label{lem_property_g}
    We have 
    \begin{enumerate}
        \item Assume that $a_i \ge 3$ then $g(\a) = g(\a - 3 \e_i) + 1$.
        \item Assume that $a_i \ge 2$ then $g(\a) \le 1 + g(\a - 2 \e_i).$
    \end{enumerate}
\end{lem}
First, we compute the depth of a starlike tree.
\begin{lem}\label{lem_depth_starlike}
Let $\a = (a_1, \ldots, a_k) \in \NN^k$ be a tuple of positive integers. Denote $T_\a$ the starlike trees obtained by joining paths of length $a_1, \ldots, a_k$ at the root $1$. Then 
$$\depth S/I(T_\a) = g(\a).$$
\end{lem}
\begin{proof} We prove by induction on $s = |\a - \mathbf{1}|$. For simplicity of notation, we denote $I = I(T_\a)$. When $s = 0$, then $\depth S/I = 1$. Assume that $s > 0$ and $a_1$ is largest, so $a_1 > 1$. We may assume that $x$ is a leaf and $y$ is a near leaf of the first branch. First, consider the case $a_1 \ge 3$. In this case, by induction and Lemma \ref{lem_property_g}, we have 
\begin{align*}
    \depth S/(I,y) &= 1 + g(\a - 2 \e_1) \ge g(\a) \\
    \depth (S/I:y) &= 1 + g(\a - 3 \e_1) = g(\a).
\end{align*}
The conclusion then follows from Lemma \ref{lem_depth_colon} and Lemma \ref{lem_upperbound}.
Now, assume that $a_1 = 2$. Assume that $a_i = 2$ for $i = 1, \ldots, \ell$  and $a_j = 1$ for $\ell < j \le k$. Let $y$ be the root of $T_\a$. We have 
\begin{align*}
    \depth S/(I,y) &= k \\
    \depth (S/I:y) &= 1 + \ell.
\end{align*}
If $\ell < k$ then by Lemma \ref{lem_depth_colon} and Lemma \ref{lem_upperbound} $\depth S/I = 1 + \ell = g(\a)$. If $\ell = k$, we prove that $\depth S/I = k$. We have $S = \k[x_1, \ldots, x_{2k+1}]$ and $I = x_1(x_2, \ldots, x_{2k}) + (x_{2i}  x_{2i+1} \mid i = 1, \ldots, k)$. Then we have $I: x_2 = (x_1,x_3) + (x_{2j} x_{2j+1} \mid j = 2, \ldots, k)$. Hence, $\depth S/(I:x_2) = k$. The conclusion follows.
\end{proof}

Before studying the depth of powers of starlike trees, we give a general procedure for bounding the depth of powers of trees. Assume that the edges of a tree $T$ are labeled $e_1, \ldots, e_{n-1}$ such that $e_i$ is a leaf of the tree $T_i$ with edges $e_i, \ldots, e_{n-1}$. We denote by $H_i$ the ideal generated by $e_1, \ldots, e_{i}$. By abuse of notation, we also denote $T_i$ the ideal generated by $e_i, \ldots, e_{n-1}$. By repeated use of Lemma \ref{lem_depth_colon}, we deduce that
\begin{lem}\label{lem_a1} Let $T$ be a tree. Let $t \ge 2$ be an integer. Then $\depth S/I^t$ belongs to the set 
$$\{ \depth S/(I(T)^t + H_\ell):e_{\ell+1} \mid \ell = 0, \ldots n-2\}.$$    
\end{lem}
Note that $I(T)^t + H_{\ell} = T_{\ell+1}^{t} + H_{\ell}$ and since $e_{\ell + 1}$ is a leaf edge of $T_{\ell+1}$ we deduce that 
\begin{equation}
(I(T)^t + H_{\ell}) : e_{\ell+1} = T_{\ell+1}^{t-1} + (H_{\ell} : e_{\ell+ 1}).    
\end{equation}

From that, we can deduce a lower bound on the depth of powers of trees by induction. Now, let $\a = (a_1, \ldots,a_k)$ be a tuple of positive integers. We assume that $a_1 \ge a_2 \ge \cdots \ge a_k$. We label the edges of the starlike tree $T_\a$ by going from the outside of each branch to the center from one branch to another. Assume that $e_{\ell +1} = uv$ is the $b$th edge of the $i$th branch. Then $H_{\ell}$ is the union of a starlike tree on the $i-1$ branches of $T$ with a path of length $b-1$. $T_{\ell+1}$ is the starlike tree $T_{\a'}$ with $\a' = (0, \ldots, 0, a_i - b+1, a_{i+1}, \ldots, a_k)$. For simplicity of notation, we assume that $u$ is farther from the center than $v$. We now give a formula for the colon ideal $H_{\ell} : e_{\ell + 1}$. 

\begin{lem}\label{lem_colon_starlike} We have
\begin{enumerate}
    \item Assume that $u$ is the leaf of the $i$th branch and $v$ is the center. Let $x_1, \ldots, x_{i-1}$ be the unique neighbours of $v$ on the first $i-1$ branches of $T$. Then $H_{\ell} : e_{\ell + 1} = H_{\ell} + (x_1, \ldots, x_{i-1})$. 
    \item Assume that $u$ is a leaf of the $i$th branch and $v$ is not the center. Then $H_{\ell} : e_{\ell+1} = H_{\ell}$. 
    \item Assume that $u$ is not a leaf and $v$ is not the center. Let $w$ be the unique neighbor of $u$ other than $v$. Then 
    $H_{\ell} : e_{\ell+1} = H_{\ell} + (w)$.
    \item Assume that $u$ is not a leaf and $v$ is the center. Let $x_1, \ldots, x_{i-1}$ be the unique neighbors of $v$ on the first $i-1$ branches of $T$ and $w$ be the unique neighbor of $u$ other than $v$. Then $H_{\ell} : e_{\ell + 1} = H_{\ell} + (x_1,\ldots, x_{i-1}, w)$. 
\end{enumerate}
\end{lem}
\begin{proof}
    These statements follows from the definition of $H_\ell$ and the fact that for a monomial ideal $I$ generated by $f_1, \ldots, f_t$ and a monomial $g$, then $I:g$ is generated by $f_1/\gcd(f_1,g), \ldots, f_t/ \gcd (f_t,g).$
\end{proof}
Let $\t \in \NN^k$ be another tuple of non-negative integers. We write $\t \ll \b$ if $t_i \le b_i$ for all $i$. Throughout, we assume that $\t \ll \a - \mathbf{1}$. We assume that $a_1 \ge a_2 \ldots \ge a_k$. First, we prove a lower bound for the depth of powers of starlike trees. 
\begin{lem}\label{lem_lowerbound_starlike} Let $\a = (a_1, \ldots, a_k) \in \NN^k$ be a tuple of positive integers. Assume that $2 \le t < |\a - \mathbf{1}|$, then 
$$ \depth S/I(T_\a)^t \ge \min ( g(\a-\t) \mid \t \in \NN^k, \t \ll \a - \mathbf{1}, |\t| = t-1).$$    
\end{lem}
\begin{proof} We use the notation from Lemma \ref{lem_a1}. We will prove by induction on $t$ that there exists a $\t_0$ such that $\t_0 \ll \a - \mathbf{1}$, $|\t_0| = t-1$ and $\depth S/ I(T_\a)^t + H_\ell \ge g(\a - \t_0)$. By induction and Lemma \ref{lem_depth_colon}, it suffices to prove that for each $\ell$, there exists such a $\t_0$ and $\depth (S/ (I(T_\a)^t + H_\ell) : e_{\ell+1}) \ge g(\a - \t_0)$. For simplicity of notation, we denote $J = (I(T_\a)^t + H_\ell) : e_{\ell+1}$, where $e_{\ell+1} = uv$. Let $S'$ be the polynomial on the variables corresponding to vertices of $T_{\ell+1}$. By Lemma \ref{lem_colon_starlike}, there are four cases as follows.

\noindent \textbf{Case 1.} $u$ is a leaf and $v$ is the center. Then $J = I (T_{\ell+1})^{t-1} + K + (x_1,\ldots,x_{i-1})$, where $T_{\ell+1}$ is the starlike tree of size $1,a_{i+1}, \ldots, a_k$ and $K$ is the unions of paths of lengths $a_j - 2$ for $j = 1, \ldots, i-1$. With our assumption, $a_j = 1$ for all $j \ge i$. Hence, by Theorem \ref{thm_mixed_sum} and Lemma \ref{lem_depth_path_cycle},
    \begin{equation}
        \depth S/J = \sum_1^{i-1} \lceil \frac{a_j - 1}{3} \rceil + 1 = \depth S/I(T_\a) =  g(\a).
    \end{equation}

\noindent \textbf{Case 2.} $u$ is a leaf and $v$ is not the center. Then $J = I(T_\a)^{t-1} + H_{\ell}$. Thus, the conclusion follows from induction on $t$ as $g(\a) \le g(\a')$ if $\a \ll \a'$.

    \noindent \textbf{Case 3.} $u$ is not a leaf and $v$ is not the center. In particular, $b \ge 2$. Then $J = I(T_{\ell + 1})^{t-1} + (w) + K$ where $T_{\ell + 1}$ is the starlike tree of size $a_1, \ldots, a_{i-1}, a_i - b +1, a_{i+1}, \ldots, a_k$ and $K = K_1 + K_2$ with $K_1$ is the edge ideal of the starlike tree which is the union of the first $i-1$ branches of $T_\a$, $K_2$ is a path of length $\max(b-3,0)$. By induction, we have
    \begin{equation}
 \depth S/J = \lceil \frac{b-2}{3} \rceil + \depth (S'/ (T_{\ell+1}^{t-1} + K_1)) \ge  \lceil \frac{b-2}{3} \rceil + g(\a  - (b-1) \e_i -\t'),
    \end{equation}
    for some $\t' \ll \a - (b-1) \e_i - \mathbf{1}$ and $|\t'| = t-2$. Let $\t_0 = \t' + \e_i$. By Lemma \ref{lem_sum_fraction},
     $$\lceil \frac{b-2}{3} \rceil + \lceil \frac{a_i - (b-1) - t_i'}{3} \rceil \lceil  \ge \lceil \frac{a_i - t_i}{3} \rceil.$$
     Hence, $\depth S/J \ge g(\a - \t_0)$.
    
    \noindent \textbf{Case 4.} $u$ is not a leaf and $v$ is the center. Then $J = T_{\ell + 1}^{t-1} + K + (x_1, \ldots, x_{i-1}, w)$ where $K$ is the union of paths of length $a_j - 2$ for $j = 1, \ldots, i-1$ and a path of length $\max(a_i - 3,0)$ and $T_{\ell + 1}$ is the starlike tree of size $1, a_{i+1}, \ldots, a_k$. Hence, 
    \begin{equation}
        \depth S/J = \sum_{j=1}^{i-1} \lceil \frac{a_j-1}{3} \rceil + \lceil \frac{a_i-2}{3} \rceil + \depth (S'/ T_{\ell+1}^{t-1}).
    \end{equation}
    Let $\a' = (1,a_{i+1},\ldots,a_k) \in N^{k-i+1}$. By induction, there exists $\t' \ll \a' - \mathbf{1}$ with $|\t'| = t-2$ such that $\depth (S'/ T_{\ell+1}^{t-1} ) \ge g(\a' - \t')$. Let $\t_0 = \t' + \e_i$. Hence, $\depth S/J \ge g(\a - \t_0)$. The conclusion follows.
\end{proof}

The following lemma is essential for establishing an upper bound for the depth of powers of edge ideals of trees.
\begin{lem}\label{lem_upperbound_pd_weakly} Let $G$ be a weakly chordal graph. Let $\B = \{B_1, \ldots, B_s\}$ be a strongly disjoint family of complete bipartite graphs of $G$ with $e_1, \ldots, e_g$ the induced matching of $G$ with $e_i \in B_i$. Let $H$ be the restriction of $G$ to $[n] \setminus \cup_{i=1}^s V(B_i)$. Assume that $H$ is non-empty, has no isolated vertices, and $N(e_i) \cap V(H) = \emptyset$ for all $i =1,\ldots, s$. Then 
$$\depth S/I(G) \le s + \depth R/I(H),$$
where $R$ is the polynomial ring on the vertices of $H$.    
\end{lem}
\begin{proof}
    Since $H$ is an induced subgraph of $G$, $H$ is weakly chordal. By Theorem \ref{thm_pd_weakly_chordal}, there exists a strongly disjoint family of complete bipartite graphs $B_{s+1}, \ldots, B_g$ of $H$ such that 
    $\pd (R/I(H)) = \sum_{i=s+1}^g |V(B_i)| - (g-s)$. Now, with the assumption of the lemma, we see that $B_1, \ldots, B_g$ form a strongly disjoint family of complete bipartite graphs of $G$. Hence, 
    $\pd (S/I(G)) \ge (\sum_{i=1}^s |V(B_i)|) - s + \pd (R/I(H)).$
    The conclusion follows from the Auslander-Buchsbaum formula.
\end{proof}
We are now ready to give an upper bound for the depth of powers of edge ideals of starlike trees. We will use the following notation. Assume that the edges of the $i$th branch of $T_\a$ are labeled $e_{i,1}, \ldots, e_{i,a_i}$ where $e_{i,1}$ is the leaf and $e_{i,a_i}$ is the one that contains the root. 

\begin{lem}\label{lem_upper_bound_starlike} Let $\t \in \NN^k$ be a tuple such that $\t \ll \a - \mathbf{1}$. Let 
$$f = \prod_{i=1}^k \prod_{j=2}^{t_i+1} e_{i,j}$$ 
be a product of $t_i$ edges from the $i$th branch. Then 
$$\depth S/(I(T_\a)^t : f) \le g(\a - \t),$$
where $t = |\t| + 1$.
\end{lem}
\begin{proof} For this lemma, we label the vertices of the $i$th branch $T_\a$ by $x_{i,1}, \ldots, x_{i,a_i+1}$ where $x_{i,1}$ is the leaf and $x_{i,a_i+1} = 1$ is the root.

First assume that $t_i \le a_i - 3$ for all $i$. Then $I(T_\a)^t : f = I(G)$ where $G$ is the union of complete bipartite graphs $B_i$ on $\{x_{i,1}, \ldots, x_{i,t_i+3}\}$ if $t_i > 0$ and the starlike tree $T_{\a'}$ where $a'_i = a_i - t_i - 2$ if $t_i > 0$. Let $\b$ be defined by $b_i = a_i - t_i - 3$ if $t_i > 0$. Then $\{B_i \mid t_i > 0\}$ and $T_\b$ is disjoint. For simplicity, we assume that $t_i > 0$ for $i = 1, \ldots, \ell$ and $t_i = 0$ for $i = \ell+1, \ldots, k$. By Lemma \ref{lem_upperbound_pd_weakly} and Lemma \ref{lem_depth_starlike}, we have     
    $$\depth S/I(G) \le \ell + g(\b) = g(\a - \t).$$

Now assume that there are some $t_i$ such that $t_i = a_i - 2$ and $t_i < a_i - 1$ for all $i$. After rearranging $a_i$, assume that $t_i = a_i -2$ for $i = 1, \ldots, s$ and $0 < t_i \le a_i - 3$ for $s + 1 \le i \le \ell$. Then $I(T_\a)^t : f = I(G)$ where $G$ is the union of complete bipartite graphs $B_i$ on $\{x_{i,1},\ldots, x_{i,a_i+1}\}$ for $i = 1, \ldots, s$ and $B_j$ on $\{ x_{j,1}, \ldots, x_{j,t_j+3}\}$ for $j = s+1, \ldots, \ell$. Let $B_i' = B_i \setminus \{x_{i,a_i+1}\}$ for $i = 1, \ldots, s$. Then $B_1', \ldots, B_s', B_{s+1}, \ldots, B_\ell$ and $T_\b$ are disjoint, where $\b = (0,\ldots,0,a_{s+1} - t_{s+1} - 3, \ldots, a_\ell - t_\ell - 3, a_{\ell+1}, \ldots, a_k)$. By Lemma \ref{lem_upperbound_pd_weakly} and Lemma \ref{lem_depth_starlike}, we have 
    $$\depth S/I(G) \le \ell  + g(\b) = g(\a-\t).$$

Finally, assume that there are some $t_i$ such that $t_i = a_i - 1$. Assume that $0 < t_i = a_i - 1$ for $i = 1, \ldots, s$, $0< t_j = a_j - 2$ for $j = s+1, \ldots, u$, $0<t_j \le a_j - 3$ for $j = u+1, \ldots, v$ and $t_j = 0$ for $j \ge v+1$. By Lemma \ref{lem_colon_product_edges}, $I(T_\a)^t : f = I(G)$ where $G$ is the union of complete bipartite graphs $B_1, B_{s+1}, \ldots, B_v$ and paths $P_{u+1}, \ldots, P_k$, where the graphs $B_i$ and paths $P_j$ are:
\begin{enumerate}
    \item $B_1$ consists of vertices $\{1,x_{1,1}, \ldots, x_{1,a_1}, x_{2,1}, \ldots, x_{2,a_2}, \ldots, x_{s,1}, \ldots, x_{s,a_{s}}, x_{j,a_j} \mid j = s+1,\ldots, k\}$,
    \item For $j = s+1, \ldots, u$, $B_j$ consists of vertices $\{x_{j,1}, \ldots, x_{j,a_j-1}\}$,
    \item For $m = u+1, \ldots, v$, $B_m$ consists of vertices $\{x_{m,1}, \ldots, x_{m,b_m}\}$, where $b_m = \min (t_m + 3,a_m -1)$.
    \item For $j = u+1, \ldots, v$, $P_j$ is a path of length $p_j = \max(a_m - t_m - 4,-1)$,
    \item For $m = v+1, \ldots, k$, $P_m$ is a path of length $p_m = \max(a_m - 2, - 1)$,
    \end{enumerate} 
    where by a path of length $-1$ we mean it does not exist. By Lemma \ref{lem_upperbound_pd_weakly} and Lemma \ref{lem_depth_path_cycle}, we deduce that 
$$\depth S/I(G) \le 1 + v - s + \sum_{m=u+1}^k \lceil \frac{p_m + 1}{3} \rceil \le g(\a  - \t).$$
The conclusion follows.
\end{proof}
It remains to compute $\min (g(\a - \t \mid \t \ll \a - \mathbf{1}, |\t| = t-1))$ in terms of $\a$ and $t$. To do so, we note the following properties of $g$. 
\begin{lem}\label{lem_minimal_values} Let $\b = (a_3, \ldots, a_k)$. We have 
\begin{enumerate}
    \item $g(3k+1,3l + 3, \b) \le g(3k+2,3l+2,\b)$.
    \item $g(3k+1,3l+1,\b) \le g(3k,3l+2,\b)$.
    \item $g(3k-2,3l+2,\b) \le g(3k,3l,\b)$.
\end{enumerate}    
\end{lem}
\begin{proof}
    These properties follow easily from the definition of $g$. We prove one of them for completeness. For (1), we have
    $$g(3k+1, 3l+3,\b) = 1 + k + l+1 + \sum_{i=3}^k \lceil \frac{a_i - 1}{3} \rceil.$$
    Furthermore,  
    $$g(3k+2,3l+2,\b) = \epsilon + k+1 + l+1 + \sum_{i=3}^k \lceil \frac{a_i-1}{3}\rceil,$$
    where $\epsilon = 1$ if $a_j = 1 \pmod 3$ for some $j \ge 3$ and $0$ otherwise. In any cases, we have $g(3k+2,3l+2,\b) \ge g(3k+1,3l+3,\b)$.    
\end{proof}
\begin{lem}\label{lem_computation_a} Let $\a = (a_1, \ldots, a_k) \in \NN^k$ be a tuple of positive integers. Let $\alpha_i$ be the number of $a_j$ such that $a_j = i \pmod 3$. We may assume that $a_j = 2 \pmod 3$ for $j = 1, \ldots, \alpha_2$, $a_j = 0 \pmod 3$ for $j = \alpha_2+1, \ldots, \alpha_2 + \alpha_0$ and $a_j = 1 \pmod 3$ for $j = \alpha_0 + \alpha_2+1, \ldots, k$. Let $\beta_1 = \min (\alpha_2,t-1)$, $\beta_2 = \min(\alpha_0, \lfloor \frac{\max(0,t-1-\alpha_2)}{2} \rfloor)$ and $\beta_3 = \lfloor \frac{\max (0, t-1 - \beta_1 - 2 \beta_2)}{3} \rfloor$. We define $\b \in \NN^k$ as follows. 
$$b_i = \begin{cases}
 a_i - 1 & \text{ for }  i = 1, \ldots, \beta_1,\\
 a_i -2 & \text{ for } i = \alpha_2+1, \ldots, \alpha_2 + \beta_2.
 \end{cases}$$
Then  
$$\min ( g(\a - \t) \mid \t \ll \a -\mathbf{1} \text{ and } |\t| = t-1) = g(\b) - \beta_3.$$
\end{lem}
\begin{proof}By Lemma \ref{lem_minimal_values}, we can choose an optimal $\t$ such that $t_j = a_j - 1 \pmod 3$ for all but one value of $j$. By Lemma \ref{lem_property_g}, the optimal $\t$ is obtained when we get the maximal number of indices $j$ such that $t_j = a_j -1 \pmod 3$. Let $\gamma_i$ be the number of $t_j$ such that $t_j = j \pmod 3$. Then, we have $\gamma_1 \le \alpha_2$, $\gamma_2 \le \alpha_0$ and $\gamma_1 + 2 \gamma_2 \le t-1$. The conclusion follows. Thus, the optimal values for $\gamma$s are precisely the $\beta$s.
\end{proof}

We are now ready for the main result of this section. 
\begin{thm}\label{depth_star_like_trees_2}
Let $\a = (a_1, \ldots, a_k) \in \NN^k$ be a tuple of positive integers. Denote $T_\a$ the starlike trees obtained by joining paths of length $a_1, \ldots, a_k$ at the root $1$. With the notation as in Lemma \ref{lem_computation_a}, for all $t$ such that $1 \le t \le |\a| - k = s$,
$$\depth S/I(T_\a)^t = g(\b) -\beta_3.$$
\end{thm}
\begin{proof} Follows from Lemmas \ref{lem_a1}, \ref{lem_lowerbound_starlike}, \ref{lem_upper_bound_starlike}, and \ref{lem_computation_a}.
\end{proof}

\begin{exm}Let $\a = (3,4,5)$. Then $\alpha_0 = 1$, $\alpha_1 = 1$ and $\alpha_2 = 1$. By Theorem \ref{depth_star_like_trees_2}, we see that the sequence $\{\depth (S/I(T_\a)^t) \mid 1 \le t \le 10\}$ is $\{5,4,4,3,3,3,2,2,2,1\}$.     
\end{exm}
Finally, we compute the depth of powers of a caterpillar tree to show that depth of powers of edge ideals of trees could drop arbitrarily at a step. 

\begin{prop}\label{prop_caterpillar} Let $G$ be a caterpillar tree based on a path of length $2$ on $3$ vertices $1, \ldots, 3$. Assume that $d_i$ is the number of leaves at $i$. We have 
$$\depth S/I(G)^t = \begin{cases} 
\min (d_1 + d_3+1,d_2+2) & \text{ if } t = 1\\
\min (d_1+1,d_2+2,d_3+1) & \text{ if } t = 2\\
1 & \text{ if } t \ge 3.\end{cases}$$       
\end{prop}
\begin{proof}
   We have $I = x_1(u_1, \ldots, u_{d_1}) + x_2 (v_1, \ldots, v_{d_2}) + x_3 (w_1, \ldots, w_{d_3}) + (x_1x_2,x_2x_3)$. Thus, $I:x_2 = (x_1,x_3,v_1, \ldots, v_{d_2})$. Hence, $\depth S/(I:x_2) = d_2+2.$ Also, $I:x_1x_3 = (u_1,\ldots,u_{d_1},w_1,\ldots,w_{d_3},x_2)$. Hence, $\depth S/(I:x_1x_3) = d_1 + d_3 + 1$. Hence, $\depth S/I \le \min (d_1 + d_3 + 1,d_2 + 2)$. 

Furthermore, we have $\depth S/(I,x_2) = d_1 + d_3 + 1$. Hence, by Lemma \ref{lem_depth_colon} we have $\depth S/I = \min (d_1 + d_3 + 1,d_2 + 2)$. 
   
   Now, we prove the formula for $\depth S/I^2$. First, we have $\depth S/I^2 \le \depth S/I \le d_2 + 2$. Now, consider $J = I^2:x_1x_2$. Let $B_1$ be a complete bipartite graph on $\{x_1,x_2,u_1,\ldots,u_{d_1},v_1,\ldots,v_{d_2}\}$. Then, we have $\pd S/J \ge d_1 + d_2 + 2$. Hence, $\depth S/I^2 \le \depth S/J \le d_3 + 1$. Similarly, we have $\depth S/I^2 \le \depth S/(I^2:x_2x_3) \le d_1 + 1$. 

Now, we prove that $\depth S/I^2 \ge \min (d_1 + 1,d_3 + 1,d_2 + 2).$ We label the edges of $T_\a$ as follows, $x_1u_1, \ldots, x_1 u_{d_1},x_3w_1,\ldots,x_3w_{d_3},x_1x_2,x_2v_1,\ldots,x_2v_{d_2},x_2x_3$. By Lemma \ref{lem_a1}, it suffices to prove that $\depth S/(I^2 + H_\ell) : e_{\ell+1} \ge \min (d_1 + 1,d_3 + 1,d_2+2)$. Let $J = (I^2 + H_\ell) : e_{\ell+1}$. There are several cases as follows. 

\noindent \textbf{Case 1.} $e_{\ell+1} = x_1 u_j$. Then $J = (I^2 + H_\ell) : e_{\ell+1} = (u_1, \ldots, u_{j-1}) + I(T_\a')$ where $T_\a'$ is the caterpillar tree with $d_1' = d_1 - j+1$. Thus, $\depth S/J = \min (d_2 + 2,d_1 - j + 1 + d_3) \ge \min (d_2+2,d_3 + 1)$.

\noindent \textbf{Case 2.} $e_{\ell+1} = x_3 w_j$. Then $J = (I^2 + H_\ell) : e_{\ell+1} = (w_1,\ldots,w_{j-1}) + I(T_\a')$ where $T_\a'$ is the caterpillar tree with $d_3 = d_3 - j + 1$. Thus, $\depth S/J = \min(d_2+2,d_1 + d_3 - j +1 ) \ge \min (d_2+2,d_1+1)$. 

\noindent \textbf{Case 3.} $e_{\ell + 1} = x_1x_2$. Then $J = (I^2 +H_\ell) : e_{\ell+1} = (u_1,\ldots,u_{d_1}) + I(T_\a')$ where $T_\a'$ is the caterpillar tree on $x_2x_3$ with $d_2' = d_2 + 1$. Hence, $\depth S/J = \min (d_2 + 2,d_3+1)$. 

\noindent \textbf{Case 4.} $e_{\ell+1} = x_2v_j$. Then $J = (x_1,v_1,\ldots,v_{j-1})$. Hence, $\depth S/J  = d_1 + \min (d_2-j+2,d_3+1) > d_3 + 1$. 

\noindent \textbf{Case 5.} $e_{\ell+1} = x_2x_3$. Then $J = (x_1,v_1,\ldots,v_{d_2},w_1,\ldots,w_{d_3})$. Hence, $\depth S/J = d_1 +1$. The conclusion follows.   
\end{proof}


\begin{thebibliography}{2}




\bibitem[BC1]{BC1}
S. Balanescu, M. Cimpoeas, 
{\em Depth and Stanley depth of powers of the path ideal of a path graph}, arXiv:2303.01132.


\bibitem[BC2]{BC2}
S. Balanescu, M. Cimpoeas, 
{\em On the depth and Stanley depth of powers of the path
ideal of a cycle graph}, arXiv:2303.15032.


\bibitem[B]{B}	
A. Banerjee, 
{\it The regularity of powers of edge ideals},
J. Algbr. Comb. {\bf 41} (2015), no. 2, 303--321.
	

\bibitem[BH]{BH}
W. Bruns and J. Herzog,
\emph{Cohen-Macaulay rings. Rev. ed.}.
Cambridge Studies in Advanced Mathematics {\bf 39}, Cambridge University Press (1998).



\bibitem[Br] {Br} 
M. Brodmann, {\em The asymptotic nature of the analytic spread}, 
Math. Proc. Cambridge Philos. Soc. {\bf 86} (1979), 35--39.

 

 
\bibitem[CHHKTT]{CHHKTT}
G. Caviglia, H. T. Ha, J. Herzog, M. Kummini, N. Terai, and N. V. Trung,
\emph{Depth and regularity modulo a principal ideal},
J. Algebraic Combin. {\bf 49} (2019), no.1, 1--20.



\bibitem[D]{D}
 R. Diestel, \emph{Graph theory, 2nd. edition,} Springer: Berlin/Heidelberg/New York/Tokyo, 2000.

 


\bibitem[HH]{HH}
J. Herzog, T. Hibi,
{\em The depth of powers of an ideal},
 J. Algebra {\bf 291} (2005), 534--550.


\bibitem[HNTT]{HNTT}
H. T. Ha, H. D. Nguyen, N.V. Trung, and T. N. Trung, 
{\em Depth functions of powers of homogeneous ideals}, Proc. Amer. Math. Soc. {\bf 609} (2021), 120--144.

 

\bibitem[K]{K}
K. Kimura, 
{\em Non-vanishingness of Betti numbers of edge ideals and complete bipartite subgraphs}, Commun.
Algebra {\bf 44} (2016), 710--730.

\bibitem[KTY] {KTY} 
K. Kimura, N. Terai and S. Yassemi,
{\em The projective dimension of symbolic powers of the edge ideal of a very well-covered graph},
 Nagoya Math. J.  {\bf 230} (2018), 160--179.
 

\bibitem [MTV]{MTV} 
N. C. Minh, T. N. Trung, and T. Vu,
{\it Depth of symbolic powers of edge ideals of cycles and stability index of depth of symbolic powers of edge ideals}, preprint.


\bibitem[MV] {MV} 
N. C. Minh and  T. Vu, {\em Survey on regularity of symbolic powers of an edge ideal}, In: Peeva, I. (eds) Commutative Algebra. Springer, Cham. (2021).

\bibitem[Mo] {Mo} 
S. Morey,
{\em Depths of powers of the edge ideal of a tree},
Comm. Algebra {\bf 38} (2010), 4042--4055.

	
\bibitem[NV1] {NV1} 
H. D. Nguyen and T. Vu,
{\em Linearity defect of edge ideals and Froberg’s theorem,} 
	J. Algbr. Comb. {\bf 44} (2016), 165--199.


\bibitem[NV2] {NV2} 
H. D. Nguyen and T. Vu,
{\em Powers of sums and their homological invariants,} 
	J. Pure Appl. Algebra {\bf 223} (2019), 3081--3111.

 
\bibitem[R]{R} 
A. Rauf, 
{\em Depth and sdepth of multigraded modules}, 
Comm. Alg. {\bf 38}, (2010), 773--784.


\bibitem[SF]{SF}
S. A. Seyed Fakhari,
{\em Lower bounds for the depth of the second power of edge ideals,} Collect. Math. (2023). https://doi.org/10.1007/s13348-023-00398-5

\bibitem[SVV]{SVV}
A. Simis, W. Vasconcelos, R.H. Villarreal, \emph{On the ideal theory of graphs,} J. Algebra {\bf 167}, (1994), 389--416.


\bibitem[T]{T}
T. N. Trung,
{\em Stability of depths of powers of edge ideals},
J. Algebra {\bf 452}, (2016), 157--187.


\end{thebibliography}
\end{document}